\documentclass[12pt]{article}
\usepackage{amsfonts}
\usepackage{amsthm}
\usepackage{amsmath}
\usepackage{enumerate}
\usepackage{color}
\usepackage{graphicx} 
\usepackage{caption}
\usepackage{subcaption}
\usepackage{setspace}

\usepackage{authblk}
\usepackage[titletoc,title]{appendix}

\def\inte{\int\limits}

\def\m{\mathtt{M}}
\def\f{\mathtt{F}}
\def\d{\mathtt{D}}
\def\re{{\rm Re}}

\def\cl{black}

\newtheorem{thm}{Theorem}
\newtheorem*{thm*}{Theorem}
\newtheorem{corollary}{Corollary}
\theoremstyle{definition}
\theoremstyle{definition}\newtheorem{lemma}{Lemma}

\author[$\dagger$]{Eugenio P. 
Balanzario\thanks{ebg@matmor.unam.mx}}
\affil[$\dagger$]{Centro de Ciencias Matem\'aticas,
Universidad Nacional Aut\'onoma de M\'exico,
Apartado Postal 61-3 (Xangari),
Morelia  Michoac\'an, M\'exico}
\author[$\dagger$]{Daniel 
Eduardo C\'ardenas Romero}
\author[$\dagger$]{Richar Chac\'on Serna}

\begin{document}

\title{A smooth version of
Landau's explicit formula}

\maketitle

\begin{abstract}
We present
a smooth version of Landau's
explicit formula for the von Mangoldt
arithmetical function. \textcolor{\cl}{
Assuming the validity of the 
Riemann hypothesis}, we show that
in order to determine whether a
natural number $\mu$ is a prime number,
it is sufficient to know the
location of a number 
\textcolor{\cl}{of non trivial zeros}
of the Riemann zeta function of
order \textcolor{\cl}{$\mu\log^{\frac{3}{2}}\mu$}. 
Next we use Heisenberg's
inequality \textcolor{\cl}{to 
support the conjecture} 
that this number of zeros cannot be essentially 
diminished.
\end{abstract}

Heisenberg's
inequality.

\medskip

\noindent
{\bf MathSciNet classification}:
11N37.

\section{Introduction}

In 1895 von Mangoldt \cite{Mangoldt}
gave a rigorous proof of
following explicit formula
first conjectured by Riemann,
\begin{eqnarray}\label{classical}
\sum_{n\leq x}
\Lambda(n)=x-\sum_\rho
\frac{x^\rho}{\rho}-
\frac{\zeta'(0)}{\zeta(0)}-
\frac{1}{2}\log\Big(1-\frac{1}{x^2}\Big).
\end{eqnarray}
In this formula, $\Lambda(n)$ is
the von Mangoldt arithmetical function
and the sum over
the non trivial zeros $\rho=\beta+
i\gamma$ of the Riemann zeta
function $\zeta(s)$ is understood in the
Cauchy principal value sense (see \cite{Davenport}).
The explicit formula allows us to translate
information about the Riemann zeta
zeros into information about the
distribution of prime numbers and it
is regarded as an important result
in the analytic theory of numbers. 

It is the aim of this note to consider
the following explicit formula
for the von Mangoldt function $\Lambda(n)$
and expose some of its consequences.
\begin{thm}\label{teo1}
For positive numbers $t$, $\alpha$
and $\lambda$, such that 
$\alpha\notin\mathbb{N}$, let 
\begin{eqnarray*}
w_{\alpha,\lambda}(t)=
\frac{t^{\alpha-1}}
{\lambda^\alpha\Gamma(\alpha)}
\exp\Big\{-\frac{t}{\lambda}\Big\}.
\end{eqnarray*}
Let $\mu=\alpha\lambda$. Then we have
\begin{eqnarray}\label{explict}
\sum_{n=1}^\infty
\Lambda(n)
w_{\alpha,\frac{\mu}{\alpha}}(n)=
1
-\sum_{\rho}
\frac{\Gamma(\alpha+\rho-1)}
{\kern.03cm\Gamma(\alpha)}
\Big(\frac{\mu}{\alpha}\Big)^{\rho-1}-
R(\mu,\alpha)
\end{eqnarray}
where
\[
R(\mu,\alpha) = 
\sum_{j=1}^\infty
\frac{(\alpha/\mu)^{2j+1}}
{(\alpha-1)_{2j+1}}-
\frac{(\alpha/\mu)^\alpha}{\Gamma(\alpha)}
\sum_{j=0}^\infty\frac{(-1)^j}{j!}
\Big(\frac{\alpha}{\mu}\Big)^j
\frac{\zeta'(1-\alpha-j)}
{\zeta(1-\alpha-j)}.
\]
Here, $(x)_n$ is the falling factorial.
\end{thm}

Formula (\ref{explict}) differs from
formula (\ref{classical}) in that
a weight function 
is introduced in the sum over $\Lambda(n)$.
This weight function $w_{\alpha,\lambda}(t)$
is actually the probability
density function of a gamma random
variable with mean value $\mu=\alpha\lambda$ 
and variance $\sigma^2=\alpha\lambda^2$. In case
that $\alpha=1$, we have that 
$w_{\alpha,\lambda}(t)$
reduces to a weight function that has
been used extensively in the analytic
theory of numbers. By allowing that $\alpha>1$,
we will be able to locate the probability
unit mass given by $w_{\alpha,\lambda}(t)$ at
any preassigned point $\mu$ of the positive
real line. 
Furthermore, with $\alpha$ and $\lambda$
as two free parameters,
we will be able not only to place the
bulk of the probability mass at $\mu$, but
also to control how much
this probability mass is concentrated
around this point of our interest.

When $\alpha\in\mathbb{N}$, then 
$w_{\alpha,\lambda}(t)$ can be considered
as the density function of the sum of
$\alpha$ independent exponentially 
distributed random variables. Thus,
if $\alpha$ is large, it follows from
the central limit theorem that
$w_{\alpha,\lambda}(t)$ is approximately
a bell shaped function. This observation
explains the given expression for $S(\mu)$
in the next theorem.

\begin{thm}\label{teo2}
Assume the Riemann hypothesis.
Let be given $\mu\in\mathbb{N}$.
Let $\sigma,\eta>0$ and $\theta>1$ be fixed
numbers. Let 
\begin{eqnarray}\label{ese}
S(\mu)=
\sum_{|j|\leq\eta\sigma}
\frac{\mu}{\mu+j}
\Lambda(\mu+j)
\exp\Big\{-\frac{1}{2}
\Big(\frac{j}{\sigma}\Big)^2\Big\}.
\end{eqnarray}
Let $\alpha=(\mu/\sigma)^2$ and
$\widetilde{w}_\alpha(\gamma)=
\Gamma(\alpha-1/2+i\gamma)/\Gamma(\alpha)$.
Then, as $\mu\to\infty$,
\begin{eqnarray}\label{smu}
\frac{S(\mu)}{\sigma\sqrt{2\pi}}=1-
\frac{\sqrt{\mu}}{\sigma}
\kern-.1cm
\sum_{|\gamma|\leq\mu\theta/\sigma}
\kern-.1cm
\kern-0cm
\widetilde{w}_\alpha(\gamma)
\Big(\frac{\sigma^2}{\mu}
\Big)^{i\gamma}
\kern-.1cm -R(\mu,\alpha)
+O\Big(\frac{\log\mu}
{\eta\kern.03cm
e^{\frac{1}{2}\eta^2}}
+\textcolor{\cl}{\frac{\mu^{\frac{3}{2}}
\log\mu}
{\theta\kern.03cm e^{\frac{1}{2}\theta^2}}
\Big)}.
\end{eqnarray}
\end{thm}

For the enunciation of theorem~\ref{teo2} we
have assumed the validity of the
Riemann hypothesis and it will
be convenient to assume it for the remainder
of this note without further notice.

When the sum $S(\mu)$ in equation (\ref{ese})
involves only one term (corresponding to $j=0$),
then theorem~\ref{teo2}
can be considered as
a smooth version of the explicit
formula of Landau \cite{Landau},
\[
\Lambda(x)=
-\frac{2\pi}{T}
\sum_{0<\gamma\leq T}x^\rho+R
\kern1cm\hbox{where}\kern1cm
R\ll\frac{\log T}{T}.
\]
Landau's explicit formula has received
due attention by 
number theorists ever since its
publication. In particular,
Gonek~\cite{Gonek2},
obtained a bound, uniform in $x$
and $T$, for the error term,
\[
R\ll\frac{x\log(2Tx)\log\log(3x)}{T}.
\]
Smooth versions of Landau's explicit
formula also exist in the literature. For
example (\cite{Montgomery}, page 410)
\begin{eqnarray}\label{ejercicio}
&&\frac{1}{a\sqrt{2\pi}}
\sum_{n=1}^\infty\Lambda(n)
\exp\Big\{-\frac{1}{2a^2}
\log^2\Big(\frac{x}{n}\Big)\Big\}=
K_1+K_2+K_3+K_4
\end{eqnarray}
where
\begin{gather*}
K_1=e^{\frac{1}{2}a^2}x,\kern1cm
K_2=-\sum_{\rho}
e^{\frac{1}{2}a^2\rho^2}x^\rho,
\kern1cm
K_3=\sum_{0<k<\log(x)/2a^2}
\frac{e^{2a^2k^2}}{x^{2k}},\\
\vbox{\kern.8cm}
K_4=-\frac{1}{2\pi}
\exp\Big\{-\frac{1}{2a^2}\log^2x\Big\}
\inte_{-\infty}^{+\infty}
\frac{\zeta'}{\zeta}\Big(-
\frac{\log x}{a^2}+it\Big)
e^{-\frac{1}{2}a^2t^2}\>dt.
\end{gather*}
Here we have that the
sum on the left hand side of 
equation (\ref{ejercicio}) is a
sum over the von Mangodt function
weighted with a function which in
a neighborhood of $x$ is bell shaped
(this follows by considering the Taylor
series expansion of $\log^2(x)$).
On the other hand, the sum
$K_2$ is a sum over the
zeros of the Riemann zeta function
with a weight function which
is also bell shaped.

Formula (\ref{smu}) is similar
to formula (\ref{ejercicio}) 
because on both sides of the equation
we have bell shaped weight functions.
That this is the case for the right
hand side of equation (\ref{smu})
is because of the fact that,
when $\alpha$ is large, then
$\widetilde{w}_\alpha(\gamma)$ is 
approximately bell
shaped for $|\gamma|\leq\theta\sqrt{\alpha}$.
In fact, in the forthcoming 
lemma~\ref{lema2} we show that
as $\alpha\to\infty$ and 
$\gamma\ll\sqrt{\alpha}$, we have
\begin{eqnarray}\label{niño}
|\widetilde{w}_\alpha(\gamma)|\sim
\frac{1}{\sqrt{\alpha}}
\exp\Big\{-\frac{1}{2\alpha}
\gamma^2\Big\}.
\end{eqnarray}
See \cite{Aryan} for a more recent example
of smooth versions of Landau's explicit
formula with bell shaped weight functions
on both sides of the equation.
Among other authors who have contributed
to the understanding of Landau's explicit
formula are Fujii \cite{Fujii} and
Kaczorowski \cite{Kaczorowski}. It is 
also interesting to note that in his
research on the difference of
$\pi(x)-\mathrm{li}(x)$, Lehman \cite{Lehman}
also worked with a bell shaped
weight function.

\section{Numerical computations}

In this section we expose some
consequences of theorem~\ref{teo2}.
Note that in equations (\ref{ese}) and (\ref{smu}),
the terms $\eta$ and
$\theta$ determine how many ``standard deviations"
are to be taken into account when we numerically
compute the sums with bell shaped weight
functions. \textcolor{\cl}{
However, the term involving 
$\theta$ within Landau's symbol in formula (\ref{smu})
is larger than the term involving $\eta$. Thus,
it is this term involving $\theta$,
the one that will determine
how many addends  are to be taken into
account in the sum over $|\gamma|$ in
formula (\ref{smu})}.

\textcolor{\cl}
{As a consequence of theorem~\ref{teo2},
we have that by a comparison between the
numerical values of $\log\mu$ and $S(\mu)$, as
computed by the right hand side of equation (\ref{smu}), 
we can decide whether a natural
number $\mu$ is a prime number}.

\textcolor{\cl}{
\begin{corollary}\label{cor1}
Let $\sigma=1/2$. Let 
\[
B=\sigma\sqrt{2\pi}
\bigg(1-
\frac{\sqrt{\mu}}{\sigma}
\sum_{|\gamma|\leq\mu\theta/\sigma}
\widetilde{w}_\alpha(\gamma)
\Big(\frac{\sigma^2}{\mu}
\Big)^{i\gamma}
\kern-.1cm -R(\mu,\alpha)\bigg).
\]
There exist a constant
$K$ such that if $\theta\geq K\sqrt{\log\mu}$,
then, for all sufficiently large $\mu$, 
if $B\geq(41/50)\log\mu$ then $\mu$
is a prime number.
\end{corollary}}

\textcolor{\cl}{
\begin{proof}
Let $\eta$ and be such
that $O(e^{-\frac{1}{2}\eta^2})<1/50$. 
Notice that
\begin{equation*}
\begin{split}
\frac{1}{\log\mu}
\sum_{1<|j|\leq\eta/2}
\frac{\mu}{\mu+j}\Lambda(\mu+j)
e^{-2j^2}
\leq \frac{\log(\mu+\eta)}{\log\mu}
\frac{\mu}{\mu-\eta}2\sum_{j=1}^\infty
e^{-2j^2}<
\frac{7}{25}
\end{split}
\end{equation*}
whenever $\mu$ is sufficiently large.
Thus, we have
\[
\frac{\Lambda(\mu)}{\log\mu}=
\frac{B}{\log\mu}+E+
O(\mu^{\frac{3}{2}}e^{-\frac{1}{2}\theta^2})
\kern1cm\hbox{with}\kern1cm
|E|< \frac{7}{25}+\frac{1}{50}.
\]
The $O$ term in the above equation 
is smaller that $1/50$ if 
$\theta\geq K\sqrt{\log\mu}$ for
some constant $K$. Therefore we have
\[
\Big|\frac{\Lambda(\mu)}{\log\mu}-
\frac{B}{\log\mu}\Big|\leq\frac{8}{25}.
\]
Hence, if $B/\log \mu>41/50$, then
$\Lambda(\mu)/\log\mu>1/2$ and therefore
$\mu$ is a prime number.
\end{proof}}

\textcolor{\cl}
{It follows
from corollary~\ref{cor1} that an order of
$\mu\log^{\frac{3}{2}}\mu$ zeros of 
$\zeta(s)$ are needed to determine whether a 
natural number is prime}. This last assertion 
is a consequence of the well known fact that the zero
counting function
$N(T)=\mathrm{Card}\{\gamma\in(0,T):
\zeta(\beta+i\gamma)=0\}$ is such that
\begin{eqnarray}\label{ndet}
N(T)\sim\frac{T}{2\pi}
\log\frac{T}{2\pi}
\end{eqnarray}
as $T\to\infty$,
(\cite{Chandrasekharan}, page, 36).
\textcolor{\cl}
{Furthermore, by setting $\sigma=1/\sqrt{2\log2}$,  
one can also show that
in order to determine whether a natural
number $\mu$ is such that $2\mu-1$ and
$2\mu+1$ are both prime numbers, 
the same number of zeros are sufficient as
when determining whether $2\mu$ is prime}.

In table~\ref{table1} we show the values
of quantity $B$ as described in
corollary~\ref{cor1} for distinct prime
numbers $\mu$ and by setting $\eta=3$
and $\theta=K\sqrt{\log\mu}$ for
selected values of the constant
$K$. We see from
these numerical computations that a rather
small value of $K$ is sufficient to
determine whether $\mu$ is a prime number.

\renewcommand{\arraystretch}{1.7}
\begin{table}[ht]
\begin{center}
\begin{tabular}{|c|c|c|c|c|c|}
\hline
$\mu$ & $K=0.5$ & 
$K=1.0$  &$K=1.5$ &$K=2.0$&
$S(\mu)$ 
\\ \hline
12\,553 & 7.83004 & 9.40828 & 9.43766 &
9.43772 & 9.43771\\ \hline
22\,307 & 9.22783 & 10.0031 & 10.0127 &
10.0127 & 10.0127\\ \hline
48\,611 & 9.8514 & 10.7817 & 10.7919 &
10.7919 & 10.7916
\\ \hline
\end{tabular}
\end{center}
\caption{\label{table1}
This table illustrates the performance
of the computation scheme of corollary~\ref{cor1}
for distinct values of the constant $K$.}
\end{table}

Now we might ask whether
a lesser number of zeros are sufficient to
determine when of natural number
is a prime number. 
In order to address this question, 
we recall that the variance 
of a probability distribution is a
measure of how much concentrated is
the probability mass around its mean
value. Another such measure is given
by the dispersion $\d[f]$ of a function
(not necessarily a probability density
function) defined by
\[
\d[f]=
\inte_{-\infty}^{+\infty}
(x-\bar{x})^2\kern.03cm
\frac{|f(x)|^2}{\parallel f
\parallel_2^2}\>dx\
\kern.8cm\hbox{where}\kern.8cm
\bar{x}=\inte_{-\infty}^{+\infty}
x\kern.03cm
\frac{|f(x)|^2}{\parallel f
\parallel_2^2}\>dx
\]
when $f\in\mathrm{L}^2(\mathbb{R})$. 
For such an $f(x)$, we will denote
its Fourier transform by $f^\f(t)=
\int_{-\infty}^{+\infty}
f(x)e^{-2\pi itx}dx$.
Heisenberg's inequality states that
if $\d[f]$ is small, then $\d[f^\f]$
must be large. More exactly, we have that
\begin{eqnarray}\label{heisenberg}
\d[f]
\cdot
\d[f^\f]
\geq\frac{1}{16\pi^2},
\end{eqnarray}
and this inequality holds as
an equality only in case that 
$f(x)=c\kern.03cme^{-kx^2}$ for constants
$k>0$ and $c\in\mathbb{C}$
(see \cite{Igari}, page 188).

In section~\ref{Heisenberg} we will
use Heisemberg's inequality
to prove the following theorem, which is
not negligible because, while 
$\widetilde{w}_\alpha(\gamma)$
is related to the Fourier transform of 
$w_{\alpha,\gamma}(t)$, it is not equal to it.

\begin{thm}\label{tres}
For the dispersion of
$\widetilde{w}_\alpha$ we have,
as $\alpha\to\infty$,
\[
\d[\widetilde{w}_\alpha]\geq
\Big(
\frac{\alpha}{2}-\frac{1}{4}
\Big)
\Big\{1+O\Big(
\frac{1}{\alpha^{
\frac{3}{7}}}\Big)\Big\}.
\]
\end{thm}

\textcolor{\cl}
{Also in section~\ref{Heisenberg}
we prove that
\begin{eqnarray}\label{suggest}
\d[\widetilde{w}_\alpha]=
\frac{\alpha}{2}-\frac{1}{4}.
\end{eqnarray}
It then 
follows that, asymptotically,
as $\alpha\to\infty$, the dispersion
$\d[\widetilde{w}_\alpha]$ is as small
as possible.
Thus, given a fixed $\mu$ and
a fixed number of terms in the sum for
$S(\mu)$, then 
it is natural to conjecture 
that the sum over $\gamma$
on the right hand side of formula
(\ref{smu}) is essentially as short
as it can be}. 

\begin{figure}
\begin{center}
\includegraphics[width=13.5cm]{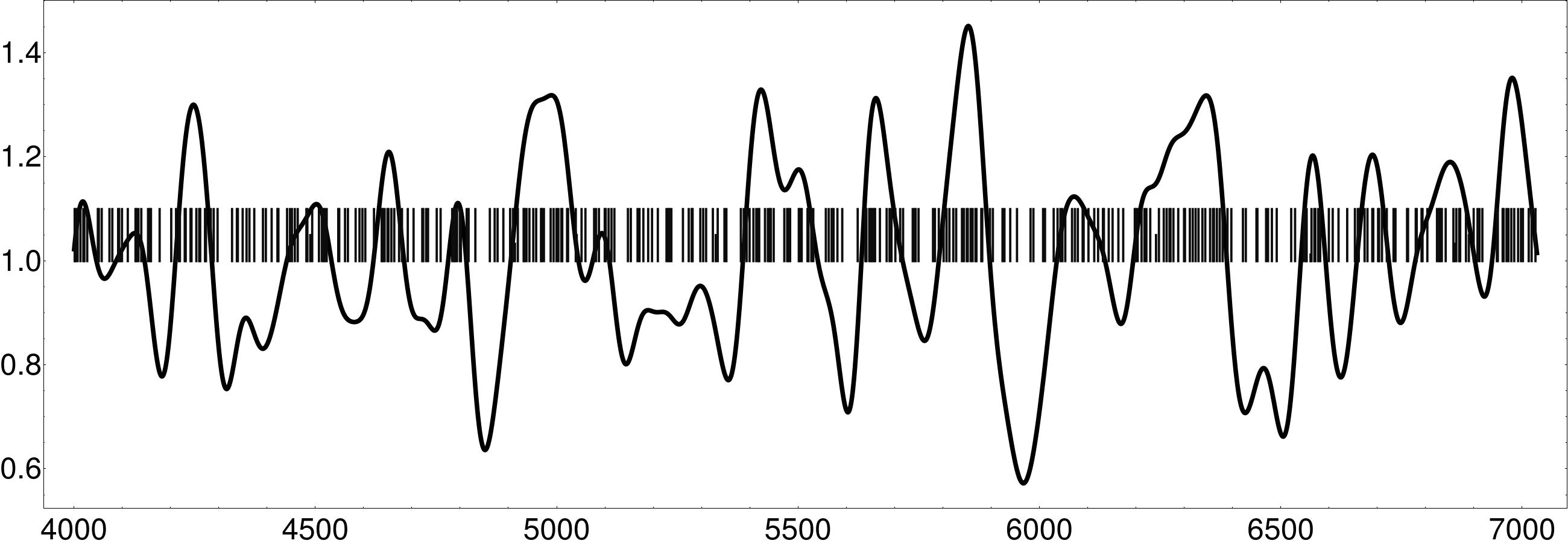}
\end{center}
\caption{The graph of $S(\mu)/\sigma\sqrt{2\pi}$
with $\sigma=25.5$. A
vertical line of height $k^{-1}$
is placed at each number of the
form $p^k$ with $p$ prime and $k\in\mathbb{N}$.}
\label{fig1}
\end{figure}

Besides the cases $\sigma=1/2$ and
$\sigma=1/\sqrt{2\log 2}$ considered
above, other choices
for $\sigma$ are interesting to consider. 
In figure~\ref{fig1}
we show the graph of $S(\mu)/\sigma\sqrt{2\pi}$
for $\mu\in(4000,7030)$ and with $\sigma=25.5$.
For the production of this graph, we used
the right hand side of equation (\ref{smu})
and the list of the Riemann zeta zeros
computed by Odlyzko \cite{Odlyzko}.
Also in figure~\ref{fig1}, a vertical line 
of height proportional to $k^{-1}$
is placed at each number of the
form $p^k$ with $p$ prime and $k\in\mathbb{N}$.
These vertical lines allow us to identify
spots along the real line
where prime numbers are abundant and
spots where prime numbers are relatively scarce. 
It is interesting to note the agreement between 
the graph of $S(\mu)/\sigma\sqrt{2\pi}$ and
the distribution of these vertical lines.
Whenever primes are more abundant than one
would expect on average, then 
$S(\mu)/\sigma\sqrt{2\pi}$ assumes values
greater than 1, which is the leading
term on the right hand side of equation (\ref{smu}).
It follows from these considerations, that
it is interesting to address 
the question of the amplitude and number,
in a given interval,
of the deviations from the leading term 
on the right hand side of formula (\ref{smu}).
We look forward to address these questions
as a further research project.

We finish this section by noticing
that the condition $\alpha\notin\mathbb{N}$
in theorem~\ref{teo1} is included in order
that the first term in the definition
of $R(\mu,\alpha)$ does not have a singular
term. It turns out that the term $R(\mu,\alpha)$
contributes negligibly to the numerical
computation of $S(\mu)/\sigma\sqrt{2\pi}$
and can therefore be ignored without detriment.

\section{Proof of theorem~\ref{teo1}}

In this section we assume, without loss
of generality, that $\alpha-1/2\in\mathbb{N}$.
For $x>0$ as a dummy variable, we let
\begin{eqnarray*}
h(x)=\sum_{n=1}^\infty
\Lambda(n)\kern.03cm
w_{\alpha,\lambda}(x\kern.01cmn).
\end{eqnarray*}
Let $\hat{h}(s)=\int_0^\infty
h(x)\kern.03cmx^{s-1}dx$ be the Mellin
transform of $h(x)$. Because of the
operational properties of the
Mellin transform, it is easy
to see that $\hat{h}(s)$ is the
product of the Dirichlet series
of $\Lambda(n)$ and the Mellin transform
of $w(x)$, that is to say,
\begin{eqnarray*}
\hat{h}(s)=-
\frac{\zeta'(s)}{\zeta(s)}
\hat{w}_{\alpha,\lambda}(s)=
-\frac{\zeta'(s)}{\zeta(s)}
\lambda^{s-1}
\frac{\Gamma(\alpha+s-1)}{\Gamma(\alpha)}.
\end{eqnarray*}
From Perron inversion formula, we have
\begin{eqnarray*}
h(1)=
\lim_{T\to\infty}
\frac{-1}{2\pi i}
\inte_{2-iT}^{2+iT}
\frac{\zeta'(s)}{\zeta(s)}
\frac{\Gamma(\alpha+s-1)}{\Gamma(\alpha)}
\lambda^{s-1}\>ds.
\end{eqnarray*}
Now we recall that there exist a sequence
of numbers $T_j$, with $j\geq2$
such that $j<T_j<j+1$ and
\begin{eqnarray*}
\Big|\frac{\zeta'(\sigma+iT_j)}
{\zeta(\sigma+iT_j)}
\Big|\ll\log(T_j)^2
\kern.8cm\text{for}\kern.8cm
-1\leq\sigma\leq2
\end{eqnarray*}
(\cite{Ingham}, page 71).
Also, in the region
obtained by removing from the
half plane $\sigma\leq-1$ the
interior of the circles
of radius 1/2 with centers at
$-2j$ with $j\in\mathbb{N}$, we have
\begin{eqnarray*}
\Big|\frac{\zeta'(x)}{\zeta(s)}
\Big|\ll\log(|s|+1)
\end{eqnarray*}
(\cite{Ingham}, page 73).
Given $\lambda=\mu/\alpha$, let
$k\in\mathbb{N}$ be such that
$k>-(3/\pi)\log\lambda$.
Let $q_j=-2j-1$ with $j\in\mathbb{N}$ and
$L=L_1\cup L_2\cup L_3\cup L_4$
be the contour of integration defined by
\begin{eqnarray*}
L_1&:&\text{the line segment
going from } 2-iT_{kj}
\text{ to } 2+iT_{kj},\\
L_2&:&\text{the line segment
going from } 2+iT_{kj}
\text{ to } -q_j+iT_{kj},\\
L_3&:&\text{the line segment
going from } -q_j+iT_{kj}
\text{ to } -q_j-iT_j,\\
L_4&:&\text{the line segment
going from } -q_j-iT_{kj}
\text{ to } 2-iT_{kj}.
\end{eqnarray*}
From the Cauchy theory of residues
$h(1)$ is equal to
\[
1
-\sum_{|\gamma|\leq T_{kj}}
\frac{\Gamma(\alpha+\rho-1)}
{\Gamma(\alpha)}
\Big(\frac{\mu}{\alpha}\Big)^{\rho-1}+
J_1+J_2+J_3
\]
where
\begin{eqnarray*}
J_1&=&-\sum_{2j\leq T_{kj}}
\frac{\Gamma(\alpha-2j-1)}
{\Gamma(\alpha)}
\Big(\frac{\mu}{\alpha}\Big)^{
-2j-1}=-
\sum_{2j\leq q_j}
\prod_{k=1}^{2j+1}
\frac{\alpha/\mu}{\alpha-k},\\
\vbox{\kern.9cm}
J_2&=&
\frac{(\alpha/\mu)^\alpha}{\Gamma(\alpha)}
\sum_{j=0}^\infty\frac{(-1)^j}{j!}
\Big(\frac{\alpha}{\mu}\Big)^j
\frac{\zeta'(1-\alpha-j)}
{\zeta(1-\alpha-j)},\\
\vbox{\kern.9cm}
J_3 &=&
\sum_{j=2}^4
\frac{1}{2\pi i}
\frac{-1}{\lambda\kern.03cm\Gamma(\alpha)}
\int_{L_j}
\frac{\zeta'(s)}{\zeta(s)}
\Gamma(\alpha+s-1)\lambda^s\>ds.
\end{eqnarray*}
The integrals over $L_2$ and
$L_4$ are bounded by
\[
\frac{\log^2(T_{kj})}{\lambda\kern.03cm\Gamma(\alpha)}
e^{-\pi T_{kj}}
\inte_{-2j-1}^2
\lambda^\sigma\>d\sigma
\ll
\frac{\log^2(kj)}{\lambda\kern.03cm\Gamma(\alpha)}
e^{-\pi kj}
\Big(\frac{1}{\lambda}\Big)^{2j+1}
\to0
\]
as $j\to\infty$ because
$k>(3/\pi)\log(1/\lambda)$.
For the estimation of the 
integral over $L_3$ it is bounded by
a constant times
\begin{eqnarray*}
&&\frac{\lambda^{q_j-1}}{\Gamma(\alpha)}
\inte_{-T_{kj}}^{+T_{kj}}
\log^2(e+|t|)
|\Gamma(\alpha-q_j-1+it)|
\>dt\\
&\ll&
\frac{\lambda^{q_j-1}}{\Gamma(\alpha)}
\Gamma(\alpha+q_j-1)
\inte_{-\infty}^{+\infty}
\log^2(e+|t|)e^{-\pi|t|}\>dt\\
\vbox{\kern.6cm}
&\ll&
\frac{1}{\lambda\kern.03cm\Gamma(\alpha)}
\Big(\frac{1}{\lambda}\Big)^
{2j+\frac{1}{2}}
\frac{1}{\Gamma(2-\alpha+2j+5/2)}.
\end{eqnarray*}
If $\alpha$ and $\lambda$
are fixed, then the last term
tends to 0 as $j\to\infty$.
This finishes the proof of
theorem~\ref{teo1}.

\section{Proof of theorem~\ref{teo2}}

\begin{lemma}\label{peter}
Let $\eta$ be a positive real
number such that
$\eta\leq 4\sqrt{\alpha}/5$. 
Then we have, as $\alpha\to\infty$,
\[
\sum_{n=1}^\infty
\Lambda(n)
w_{\alpha,\lambda}(n)=\kern-.2cm
\sum_{\sqrt{\alpha}
|n-\mu|\leq\eta\mu}\kern-.2cm
\Lambda(n)
w_{\alpha,\lambda}(n)
+O\Big(
\log(\mu)\kern.04cm 
\frac{e^{-\frac{1}{2}\eta^2}}{\eta}\Big).
\]
\end{lemma}

\begin{proof}
Let us write
\[
E_{1}=\kern-.3cm
\sum_{\sqrt{\alpha}
(n-\mu)>\eta\mu}\kern-.3cm
\Lambda(n)
w_{\alpha,\lambda}(n)
\kern.7cm\text{and}\kern.7cm
E_{2}=\kern-.3cm
\sum_{\sqrt{\alpha}
(n-\mu)<-\eta\mu}\kern-.3cm
\Lambda(n)
w_{\alpha,\lambda}(n).
\]
Let $u=\mu+
\eta\mu/\sqrt{\alpha}$. Then
\[
E_1\ll
\inte_u^\infty
\log(t)
\frac{t^{\alpha-1}}
{\lambda^\alpha\Gamma(\alpha)}
\exp\Big\{-\frac{t}{\lambda}\Big\}\>
dt =J_1+J_2
\]
where $J_1$ is the above integral
from $u$ to $9\mu/5$ and $J_2$ is
the integral from $9\mu/5$ to $\infty$.
For $J_1$ we have
\begin{equation*}
\begin{split}
J_1 \ll&\
\Big(\frac{\mu}{\lambda}\Big)^\alpha
\sqrt{\alpha}
\Big(\frac{e}{\alpha}\Big)^\alpha
\log(\mu)\inte_{1+\eta/\sqrt{\alpha}}^
{\frac{9}{5}}
t^{\alpha-1}e^{-\alpha t}\>dt\\
\vbox{\kern1.1cm}
\leq&\
\sqrt{\alpha}\kern.04cm e^\alpha
\log(\mu) \inte_{1+\eta/\sqrt{\alpha}}^
{\frac{9}{5}}t^{\alpha}e^{-\alpha t}\>dt=
\sqrt{\alpha}\kern.04cm 
\log(\mu)
\inte_{\eta/\sqrt{\alpha}}^
{\frac{4}{5}}
e^{\alpha(\log(1+t)-t)}\,dt
\\
\vbox{\kern1.1cm}
\ll&\
\sqrt{\alpha}\kern.04cm 
\log(\mu)
\inte_{\eta/\sqrt{\alpha}}^
{\infty}
e^{-\alpha\frac{1}{2}t^2}\,dt
\ll
\log(\mu)\frac{e^{-\frac{1}{2}\eta^2}}{\eta}.
\end{split}
\end{equation*}
On the other hand,
\begin{equation*}
\begin{split}
J_2 =&\
\frac{\alpha^\alpha}
{\Gamma(\alpha)}
\inte_{\frac{9}{5}(1+
\frac{\eta}{\sqrt{\alpha}})}^\infty
\log(\mu t)\kern.04cm
t^{\alpha-1}e^{-\alpha t}\,dt \ll
\sqrt{\alpha}\kern.04cm
e^\alpha\log(\mu)
\inte_{\frac{9}{5}(1+
\frac{\eta}{\sqrt{\alpha}})}^\infty
t^\alpha e^{-\alpha t}\,dt\\
\vbox{\kern1.1cm}
=&\
\sqrt{\alpha}\kern.04cm
\log(\mu)
\inte_{\frac{4}{5}+
\frac{9\eta}{5\sqrt{\alpha}}}^\infty
(1+t)^\alpha e^{-\alpha t}\,dt\ll
\sqrt{\alpha}\kern.04cm
\log(\mu)
\inte_1^\infty e^{-\frac{\alpha}{2}t}\,dt
\ll
\frac{\log(\mu)}{\sqrt{\alpha}}
e^{-\frac{1}{2}\alpha}.
\end{split}
\end{equation*}
For the estimation of 
$E_2$, we follow
the same steps as for the
estimation of $E_1$.
Let $\ell=\mu-\eta\mu/\sqrt{\alpha}$.
Then
\[
E_2\ll\log(\mu)
\inte_0^\ell
\frac{t^{\alpha-1}}{\lambda^\alpha
\Gamma(\alpha)}
\exp\Big\{-\frac{t}{\lambda}
\Big\}\>dt
\ll\log(\mu)
\sqrt{\alpha}\kern.04cm 
e^\alpha
\inte_0^{1-\eta/\sqrt{\alpha}}
t^{\alpha-1} e^{-(\alpha-1)t}\,dt.
\]
We split this last integral
in two parts: from 0 to $1/e^2$
and from $1/e^2$ to $1-\eta/\sqrt{\alpha}$.
For the first integral, we have
\[
\log(\mu)
\sqrt{\alpha}\kern.04cm 
e^\alpha
\inte_0^{1/e^2}
t^{\alpha-1} e^{-(\alpha-1) t}\,dt
\ll
\log(\mu)
\frac{\sqrt{\alpha}\kern.04cm 
e^\alpha}{e^{2\alpha}}
\inte_0^\infty
e^{-(\alpha-1) t}\,dt
\ll
\log(\mu) 
\frac{e^{-\alpha}}{\sqrt{\alpha}}.
\]
For the second integral we have,
\begin{equation*}
\begin{split}
\log(\mu)
\sqrt{\alpha}\kern.04cm 
e^\alpha
\inte_{1/e^2}^{1-\eta/\sqrt{\alpha}}
t^{\alpha-1} e^{-(\alpha-1) t}\,dt
\ll&\ 
\log(\mu)
\inte_{-\infty}^{-\eta}
e^{-\frac{1}{2}t^2}\>dt\ll 
\log(\mu)\kern.03cm
\frac{e^{-\frac{1}{2}
\eta^2}}{\eta}.
\end{split}
\end{equation*}
This finishes the proof of
the lemma.
\end{proof}

\begin{lemma}\label{ema}
\textcolor{\cl}
{Let $\theta\geq1$. Then we have,
\[
\sum_\rho
\widetilde{w}_\alpha(\gamma)
\Big(\frac{\mu}{\alpha}
\Big)^{-\frac{1}{2}+i\gamma}=
\sum_{|\gamma|\leq\theta\sqrt{\alpha}}
\widetilde{w}_\alpha(\gamma)
\Big(\frac{\mu}{\alpha}
\Big)^{-\frac{1}{2}+i\gamma}
+O\big(\alpha^{\frac{3}{4}}e^{-\frac{\theta^2}{2}}
\log\alpha
\big).
\]}
\end{lemma}

\textcolor{\cl}
{\begin{proof}
From Euler-Maclaurin sum formula we
have that
$\log|\Gamma(\alpha+it)|$ equals
\[
\Big(\alpha-\frac{1}{2}\Big)\frac{1}{2}
\log(\alpha^2+t^2)-\alpha+\frac{1}{2}
\log(2\pi)-|t|\arctan\frac{|t|}{\alpha}
+O\Big(\frac{1}{|t|}\arctan\frac{|t|}{\alpha}
\Big)
\]
(see \cite{Andrews}, page 21).
Hence, for $\sqrt{\alpha}\leq t<\alpha$, we have
\begin{equation*}
\begin{split}
|\widetilde{w}_\alpha(t)|
&\ll
\sqrt{\alpha}\Big(\frac{e}{\alpha}\Big)^\alpha
|\alpha+it|^{\alpha-1}\exp\Big\{
-\alpha-|t|\arctan\frac{|t|}{\alpha}
\Big\}\\
\vbox{\kern.7cm}
&\ll
\Big|1+i\frac{t}{\alpha}\Big|^\alpha
\exp\Big\{-|t|\arctan\frac{|t|}{\alpha}\Big\}\\
\vbox{\kern.7cm}
&\ll
\exp\Big\{\frac{\alpha}{2}\log\Big(
1+\Big(\frac{t}{\alpha}\Big)^2\Big)-
|t|\arctan\frac{|t|}{\alpha}\Big\}\\
\vbox{\kern.7cm}
&\ll
\exp\Big\{-\frac{t^2}{2\alpha}\Big\}.
\end{split}
\end{equation*}
For $t>\alpha$, we have\begin{equation*}
\begin{split}
|\widetilde{w}_\alpha(\gamma)|
&\ll\sqrt{\alpha}\Big(\frac{e}{\alpha}\Big)^\alpha
|\alpha+it|^{\alpha-1}
\exp\Big\{-\alpha-|t|\arctan\frac{|t|}{\alpha}
\Big\}\\
\vbox{\kern.7cm}
&\ll
\Big(\frac{1}{\alpha}\Big)^\alpha
|\alpha+it|^{\alpha}
\exp\Big\{-\frac{\pi}{2}|t|\Big\}\\
\vbox{\kern.7cm}
&\ll
\Big(\frac{\sqrt{2}}
{\alpha}\Big)^\alpha
|t|^\alpha
\exp\Big\{-\frac{\pi}{2}|t|\Big\}.
\end{split}
\end{equation*}
Thus, with $N(t)$ is as in (\ref{ndet}),
\begin{gather*}
\sum_{|\gamma|>\theta\sqrt{\alpha}}
\widetilde{w}_\alpha(\gamma)
\Big(\frac{\mu}{\alpha}
\Big)^{-\frac{1}{2}+i\gamma}
\ll J_1+J_2,
\end{gather*}
where, on the one hand,
\begin{equation*}
\begin{split}
J_1\ll
\sqrt{\frac{\alpha}{\mu}}
\inte^\alpha_{\theta\sqrt{\alpha}}
e^{-\frac{t^2}{2\alpha}}\>dN(t)\leq
\sqrt{\frac{\alpha}{\mu}}\log(\alpha)
\inte^\infty_{\theta\sqrt{\alpha}}
e^{-\frac{t^2}{2\alpha}}\>dt=
\frac{\alpha^{\frac{3}{4}}
\log \alpha}
{\theta\kern.03cm 
e^{\frac{1}{2}\theta^2}}.
\end{split}
\end{equation*}
On the other hand,
\begin{equation*}
\begin{split}
J_2 &\ll
\sqrt{\frac{\alpha}{\mu}}
\Big(\frac{\sqrt{2}}
{\alpha}\Big)^\alpha
\inte_\alpha^\infty
t^\alpha\kern.03cm
e^{-\frac{\pi}{2}t}\>dN(t)
\leq
\sqrt{\frac{\alpha}{\mu}}
\Big(\frac{\sqrt{2}}
{\alpha}\Big)^\alpha
\inte_0^\infty
t^{\alpha+2}\kern.03cm 
e^{-\frac{\pi}{2}t}\>dt\\
&\ll
\vbox{\kern.7cm}
\sqrt{\frac{\alpha}{\mu}}
\Big(\frac{2\sqrt{2}}
{\pi\alpha}\Big)^\alpha
\Gamma(\alpha+3)\ll
\alpha^4\Big(\frac{2\sqrt{2}}{e\kern.03cm\pi}
\Big)^\alpha.
\end{split}
\end{equation*}
This finishes the proof of the lemma.
\end{proof}}

With $x=-1/2$, the following lemma
implies that the 
relation~(\ref{niño}) holds true.

\begin{lemma}\label{lema2}
Let $x$ be a fixed real number.
Let $|y|\ll\sqrt{\alpha}$. 
If $\alpha\to\infty$, then
\[
\Big|
\frac{\Gamma(\alpha+x+iy)}
{\Gamma(\alpha)}\Big|=
\alpha^x\exp\Big\{
-\frac{y^2}{2(\alpha+x)}
\Big\}\Big\{1
+O\Big(\frac{1}{\alpha}
\Big)\Big\}.
\]
\end{lemma}

\begin{proof}
Let $z=x+iy$ and $A=\Gamma(\alpha+z)/
\Gamma(\alpha)$.
By Stirling's formula
\[
\log\Gamma(\alpha)
=\alpha\log\alpha-
\alpha+\frac{1}{2}
\log\frac{2\pi}{\alpha}+
O\Big(\frac{1}{\alpha}\Big)
\]
we have that
\begin{eqnarray*}
\log A
&=&
\log\Gamma(\alpha+z)-
\log\Gamma(\alpha)\\
&=&
\vbox{\kern.7cm}
\Big(\alpha+z-\frac{1}{2}\Big)\log(\alpha+z)
-(\alpha+z)
-\Big(\alpha-\frac{1}{2}\Big)\log(\alpha)
+\alpha+O\Big(\frac{1}{\alpha}\Big)\\
\vbox{\kern.7cm}
&=&\log\alpha^z+
\Big(\alpha+z-\frac{1}{2}\Big)
\log\Big(1+\frac{z}{\alpha}\Big)-z
+O\Big(\frac{1}{\alpha}\Big).
\end{eqnarray*}
Now we take the real part of
$\log A$,
\begin{gather*}
\re\Big[\log\alpha^z+
\Big(\alpha+z-\frac{1}{2}\Big)
\Big(\log\Big|
1+\frac{z}{\alpha}\Big|+
i\arctan\frac{y}{\alpha+x}\Big)
-z+O\Big(\frac{1}{\alpha}\Big)\Big]\\
=
\log\alpha^x+
\vbox{\kern1cm}
\Big(\alpha+x-\frac{1}{2}\Big)
\log\sqrt{\frac{(\alpha+x)^2+y^2}
{\alpha^2}}-y\arctan\frac{y}{\alpha+x}-x
+O\Big(\frac{1}{\alpha}\Big).
\end{gather*}
If $|u|<1$, then $\arctan(u)=
u+O(|u|^3)$. Hence, if $y<\alpha+x$, then
$\re[\log A]-x\log\alpha$ is equal to
(we write $E=O(1/\alpha)$)
\begin{equation*}
\begin{split}
& \vbox{\kern.8cm}
\Big(\alpha+x-\frac{1}{2}\Big)
\Big[\log\Big(1+\frac{x}{\alpha}\Big)
+\frac{y^2}{2(\alpha+x)^2}+O\Big(
\frac{y^4}{\alpha^4}\Big)\Big]-\frac{y^2}{\alpha+x}-x+E
\\
\vbox{\kern.8cm}
=\ &
\Big(\alpha+x-\frac{1}{2}\Big)
\log\Big(1+\frac{x}{\alpha}\Big)
+\frac{\alpha+x-1/2}{
2(\alpha+x)^2}y^2
-\frac{y^2}{\alpha+x}-x
+O\Big(\frac{y^4}{\alpha^3}\Big)
+E\\
\vbox{\kern.8cm}
=\ &
\Big(\alpha+x-\frac{1}{2}\Big)
\Big[\frac{x}{\alpha}+
O\Big(\frac{1}{\alpha^2}\Big)\Big]
-\frac{y^2}{2(\alpha+x)}
-\frac{y^2}{4(\alpha+x)^2}
-x+O\Big(\frac{y^4}{\alpha^3}\Big)+E\\
\vbox{\kern.8cm}
=\ &
-\frac{y^2}{2(\alpha+x)}
\bigg\{1+\frac{1}{2(\alpha+x)}\bigg\}
+O\Big(\frac{1}{\alpha}\Big)
+O\Big(\frac{y^4}{\alpha^3}\Big).
\end{split}
\end{equation*}
Since $y^4/\alpha^3\ll1/\alpha$,
then this finishes the proof
of the lemma.
\end{proof}

Now we can undertake the
proof of theorem~\ref{teo2}.
Given $\mu$, $\sigma$ and
$\alpha=(\mu/\sigma)^2$ we have
\begin{eqnarray*}
\sum_{|j|\leq\eta\sigma}
\Lambda(\mu+j)
w_{\alpha,\frac{\mu}{\alpha}}(\mu+j)
=1-
\sqrt{\frac{\alpha}{\mu}}
\sum_{|\gamma|\leq\theta\sqrt{\alpha}}
\kern-.2cm
\widetilde{w}_\alpha(\gamma)
\Big(\frac{\mu}{\alpha}
\Big)^{i\gamma}
-R(\mu\alpha)+E_1
\end{eqnarray*}
where $E_1\ll\log(\mu)e^{-\frac{1}{2}
\eta^2}/\eta+e^{-\theta\sqrt{\alpha}}$. 
Now be claim that, for
$|j|\leq\eta\sigma$, 
\begin{eqnarray}\label{claim}
w_{\alpha,\frac{\mu}{\alpha}}
(\mu+j)=\frac{1}{\mu+j}
\sqrt{\frac{\alpha}{2\pi}}\exp
\Big\{-\frac{\alpha}{2}
\Big(\frac{j}{\mu}\Big)^2\Big\}
\Big\{1+O\Big(
\frac{1}
{\sqrt{\alpha}}\Big)\Big\}.
\end{eqnarray}
Indeed, by Stirling's formula,
\begin{eqnarray*}
w_{\alpha,\frac{\mu}{\alpha}}
(\mu+j) &=&
\frac{(\mu+j)^{\alpha-1}}
{\displaystyle
\Big(\frac{\mu}{\alpha}
\Big)^\alpha\sqrt{\frac{2\pi}
{\alpha}}\Big(\frac{\alpha}{e}
\Big)^\alpha}\exp\Big\{-
\frac{\mu+j}{\mu/\alpha}\Big\}
\Big\{1+O\Big(\frac{1}{\alpha}\Big)\Big\}\\
\vbox{\kern.8cm}
&=&
\frac{1}{\mu+j}
\Big(1+\frac{j}{\mu}\Big)^\alpha
\sqrt{\frac{\alpha}{2\pi}}
\kern.03cm e^\alpha\exp\Big\{
-\alpha\Big(1+\frac{j}{\mu}\Big)\Big\}
\Big\{1+O\Big(\frac{1}{\alpha}\Big)\Big\}\\
\vbox{\kern.8cm}
&=&
\frac{1}{\mu+j}
\sqrt{\frac{\alpha}{2\pi}}\exp\Big\{
-\frac{j}{\mu}\alpha+\alpha
\log\Big(1+\frac{j}{\mu}\Big)\Big\}
\Big\{1+O\Big(\frac{1}{\alpha}\Big)\Big\}.
\end{eqnarray*}
Since $|j|\leq\eta\sigma$, then 
\begin{eqnarray*}
w_{\alpha,\frac{\mu}{\alpha}}
(\mu+j)
&=&
\frac{1}{\mu+j}
\sqrt{\frac{\alpha}{2\pi}}\exp\Big\{
-\frac{\alpha}{2}\Big[
\Big(\frac{j}{\mu}
\Big)^2+O\Big(\frac{|j|}{\mu}
\Big)^3\Big]\Big\}
\Big\{1+O\Big(\frac{1}{\alpha}\Big)\Big\}\\
\vbox{\kern.8cm}
&=&
\frac{1}{\mu+j}
\sqrt{\frac{\alpha}{2\pi}}\exp\Big\{
-\frac{\alpha}{2}\Big(\frac{j}{\mu}
\Big)^2\Big\}\Big\{1+O\Big(
\frac{1}
{\sqrt{\alpha}}\Big)\Big\}.
\end{eqnarray*}
This finishes the proof of (\ref{claim}).
Now we have
\begin{equation*}
\begin{split}
\frac{S(\mu)}{\sigma\sqrt{2\pi}} =&\
\sum_{|j|\leq\eta\sigma}
\frac{\mu}{\mu+j}
\Lambda(\mu+j)
\exp\Big\{-\frac{\alpha}{2}
\Big(\frac{j}{\mu}\Big)^2\Big\}\\
\vbox{\kern.8cm}
=&\
\Big\{1-\sqrt{\frac{\alpha}{\mu}}
\sum_{|\gamma|\leq\theta\sqrt{\alpha}}
\kern-.2cm
\widetilde{w}_\alpha(\gamma)
\Big(\frac{\mu}{\alpha}
\Big)^{i\gamma}
+E_1
\Big\}\Big\{
1+
O\Big(\frac{1}
{\sqrt{\alpha}}\Big)\Big\}\\
\vbox{\kern.8cm}
=&\
1-
\sqrt{\frac{\alpha}{\mu}}
\sum_{|\gamma|\leq\theta\sqrt{\alpha}}
\kern-.2cm
\widetilde{w}_\alpha(\gamma)
\Big(\frac{\mu}{\alpha}
\Big)^{i\gamma}+
E_1+
O\Big(\frac{1}
{\sqrt{\alpha}}\Big)+E_2
\end{split}
\end{equation*}
where
\[
E_2\ll\frac{\log^\frac{3}{2}\alpha}{\sqrt{\mu}}
\sum_{|\gamma|\leq\theta\sqrt{\alpha}}
\kern-.2cm
|\widetilde{w}_\alpha(\gamma)|\ll
\frac{\log^\frac{5}{2}\alpha}{\sqrt{\mu}}
\]
because, from lemma~\ref{lema2}, 
$|\widetilde{w}_\alpha(\gamma)|\ll
1/\sqrt{\alpha}$ and the above sum has
$N(\theta\sqrt{\alpha})$ terms, where
$N(T)$ is as in equation (\ref{ndet}).
With this estimation for $E_2$ we
finish the proof of theorem~\ref{teo2}.

\section{Heisenberg inequality}\label{Heisenberg}

In this section we first prove that 
equation (\ref{suggest})
holds true, and then, 
starting with the following lemma~\ref{lema1}, 
we use Heisenberg's inequality~(\ref{heisenberg})
in order to prove theorem~\ref{tres}.

\begin{lemma}
If $\widetilde{w}_\alpha(\gamma)$
is as in theorem~\ref{teo2}, then
$\d[\widetilde{w}_\alpha]=\alpha/2-1/4$.
\end{lemma}

\begin{proof}
From tables of cosine Fourier transforms,
we have
\[
\inte_{-\infty}^{+\infty}
|\Gamma(\alpha+it)|^2
\cos(y\kern.03cmt)\>dt=
\frac{\pi\kern.03cm\Gamma(2\alpha)}
{2^{2\alpha-1}}
\cosh^{-2\alpha}\Big(\frac{y}{2}\Big)
\]
(\cite{Oberhettinger}, page 47).
Therefore,
\[
\inte_{-\infty}^{+\infty}
|\Gamma(\alpha+it)|^2\>dt=
\frac{\pi\kern.03cm\Gamma(2\alpha)}
{2^{2\alpha-1}}
\kern.5cm\hbox{and}\kern.5cm
\inte_{-\infty}^{+\infty}t^2
|\Gamma(\alpha+it)|^2\>dt=
\frac{\alpha\kern.03cm\pi\kern.03cm
\Gamma(2\alpha)}
{2^{2\alpha}}.
\]
The quotient of these two integrals 
is equal to $\alpha/2$. We finish the
proof by writing $\alpha-1/2$ in
place of $\alpha$.
\end{proof}

\begin{lemma}\label{lema1}
Let $a=\lambda(\alpha+\sigma-1)$ and
let $|x-\log a|\leq \alpha^{-
\frac{10}{21}}$.
For $\alpha\to\infty$, we have that
\begin{eqnarray*}
w_{\alpha,\lambda}(
e^x)e^{x\sigma}=W(x)
\Big\{1+O\Big(
\frac{1}{\alpha^{
\frac{3}{7}}}\Big)\Big\}
\end{eqnarray*}
where
\[
W(x)=
\sqrt{\frac{\alpha}{2\pi}}
\kern.05cm(\lambda
\kern.03cm\alpha)^{\sigma-1}
\exp\Big\{-
\frac{\alpha+\sigma-1}{2}
\big(x-\log a\big)^2\Big\}.
\]
Moreover, we have that
\[
\d[W]=\frac{1}{2\alpha-1}.
\]
\end{lemma}

\begin{proof}
Notice first that
\begin{eqnarray}\label{pepe}
ax-e^x=
a\log\Big(\frac{a}{e}\Big)
-\frac{a}{2}
\big(x-
\log(a
)\big)^2-R
\end{eqnarray}
with $R=(1/2)\int_{\log a}^x
(x-y)^2e^y\>dy$.
If $\log a\leq
x\leq\log a+\alpha^{-\frac{10}{21}}$,
then there exists a number $\xi$
such that $\log a\leq\xi\leq x$ and
\[
|R|=
\Bigg|\frac{1}{2}(x-\xi)^2\inte_{\log a}^x
e^y\>dy\Bigg|
\ll\frac{a}{\alpha^{
\frac{20}{21}}}
(\exp\{\alpha^{-\frac{10}{21}}\}-1)\ll
\frac{\lambda}{\alpha^{
\frac{3}{7}}}.
\]
A similar bound holds true when
$\log a-\alpha^{-\frac{10}{21}}
\leq x\leq\log a$. Hence,
\[
\exp\Big\{\frac{R}{\lambda}\Big\}=
1+O\Big(
\frac{1}{\alpha^{
\frac{3}{7}}}\Big).
\]
Now,
\begin{eqnarray*}
w_{\alpha,\lambda}(
e^x)e^{x\sigma} 
&=&\frac{e^{x(\alpha+\sigma-1)}}
{\lambda^\alpha\Gamma(\alpha)}
\exp\Big\{-\frac{e^x}{\lambda}
\Big\}\\
&=&
\vbox{\kern.8cm}
\frac{1}{\lambda^\alpha\Gamma(\alpha)}
\exp\Big\{\frac{1}{\lambda}
\big(\lambda(\alpha+\sigma-1)x
-e^x\big)\Big\}.
\end{eqnarray*}
Equation~(\ref{pepe})
implies that $w_{\alpha,\lambda}(
e^x)e^{x\sigma}$ is equal to
\[
\frac{1}{\lambda^\alpha\Gamma(\alpha)}
\Big(\frac{a}{e}
\Big)^{\alpha+\sigma-1}
\exp\Big\{-\frac{\alpha+\sigma-1}
{2}\big(x-\log a
\big)^2-\frac{R}{\lambda}\Big\}.
\]
By Stirling's formula,
\begin{eqnarray*}
&&\frac{1}{\lambda^\alpha\Gamma(\alpha)}
\Big(\frac{\lambda\alpha}{e}
\big(1+\frac{\sigma-1}{\alpha}
\big)\Big)^{\alpha+\sigma-1} \\
\vbox{\kern.9cm}
&=&
\sqrt{\frac{\alpha}{2\pi}}
\Big(\frac{e}{\lambda\alpha}\Big)^\alpha
\Big(\frac{\lambda\alpha}{e}\Big)^{
\alpha+\sigma-1}
\Big(1+\frac{\sigma-1}{\alpha}\Big)^{
\alpha+\sigma-1}\Big\{1+O\Big(
\frac{1}{\alpha}\Big)\Big\}\\
\vbox{\kern.9cm}
&=&
\sqrt{\frac{\alpha}{2\pi}}
\kern.05cm(\lambda\alpha)^{\sigma-1}
\Big\{1+O\Big(
\frac{1}{\alpha}\Big)\Big\}.
\end{eqnarray*}
The expression for $\d[W]$ is
obtained from a direct calculation.
\end{proof}

\begin{lemma}\label{lemaB}
Let $a$ be as in lemma~\ref{lema1}
and let $\sigma\leq1$.
If $x\geq\log a+
\alpha^{-\frac{10}{21}}$, 
then,
\[
w_{\alpha,\lambda}(
e^x)e^{x\sigma}\ll
\sqrt{\alpha}
\exp\Big\{-\frac{\alpha}{2}
(x-\log a)^2-
\frac{\alpha^{\frac{11}{21}}}
{2}
(x-\log a)\Big\}.
\]
If $\log a-\alpha^{-\frac{10}{21}}-1<
x\leq\log a-
\alpha^{-\frac{10}{21}}$, 
then,
\[
w_{\alpha,\lambda}(
e^x)e^{x\sigma}\ll
\sqrt{\alpha}
\exp\Big\{-\frac{\alpha}{4}
(x-\log a)^2-
\frac{\alpha^{\frac{11}{21}}}
{2}
(x-\log a)\Big\}.
\]
If $x\leq\log a-\alpha^{-\frac{10}{21}}
-1$, 
then,
\[
w_{\alpha,\lambda}(
e^x)\kern.04cm e^{x\sigma}\ll
\sqrt{\alpha}\kern.04cme^{\alpha 
(x-\log a+1)}.
\]
\end{lemma}

\begin{proof}
Indeed, we use Stirling's formula
to see that
$w_{\alpha,\lambda}(
e^x)e^{x\sigma}$ is bounded by a
positive constant times
\begin{eqnarray*}
B&=&
\sqrt{\frac{\alpha}{2\pi}}\exp\Big\{
(\alpha+\sigma-1)x-\frac{1}{\lambda}
e^x+\alpha\log\Big(\frac{e}{\lambda
\alpha}\Big)\Big\}\\
\vbox{\kern.8cm}
&=&
\sqrt{\frac{\alpha}{2\pi}}\exp\Big\{
\alpha(x-\log a)
-\frac{1}{\lambda}
e^x+\alpha\Big\}
e^{x(\sigma-1)}
\Big(1+\frac{\sigma-1}{\alpha}\Big)^\alpha.
\end{eqnarray*}
Let $x=\log a+
\alpha^{-\frac{10}{21}}+z$ 
where $z\geq0$. Then, 
\begin{eqnarray*}
B &\ll&
\sqrt{\alpha}\exp\Big\{
\alpha
\big(\alpha^{-\frac{10}{21}}+
z\big)-(\alpha+\sigma-1)
\kern.03cm
e^{z+\alpha^{-\frac{10}{21}}}
+\alpha\Big\}\\
\vbox{\kern.8cm}
&\ll&
\sqrt{\alpha}\exp\Big\{
\alpha^{\frac{11}{21}}+\alpha z
-(\alpha+\sigma-1)\big(1+
\alpha^{-\frac{10}{21}}
\big)\Big(1+z+\frac{1}{2}z^2\Big)
+\alpha
\Big\}\\
\vbox{\kern.8cm}
&\ll&
\sqrt{\alpha}
\exp\Big\{-\frac{\alpha}{2}
z^2-\frac{\alpha^{\frac{11}{21}}}
{2}z\Big\}.
\end{eqnarray*}
Let $x=\log a-
\alpha^{-\frac{10}{21}}-z$ 
where $1>z\geq0$. Then,
\begin{eqnarray*}
B &\ll&
\sqrt{\alpha}\exp\Big\{-
\alpha^{\frac{11}{21}}-
\alpha z
-(\alpha+\sigma-1)\big(1-
\alpha^{-\frac{10}{21}}
\big)\Big(1-z+\frac{1}{2}z^2\Big)
+\alpha
\Big\}\\
\vbox{\kern.8cm}
&\ll&
\sqrt{\alpha}
\exp\Big\{-\frac{\alpha}{4}
z^2-\frac{\alpha^{\frac{11}{21}}}
{2}z\Big\}.
\end{eqnarray*}
Let $x=\log a-
\alpha^{-\frac{10}{21}}-1-z$ 
where $z\geq0$. Then, 
\[
B \ll
\sqrt{\alpha}\exp\Big\{
\frac{1}{2}\alpha(-1-z)+\alpha
\Big\},
\] 
and this is as stated in the lemma.
\end{proof}

In order to apply Heisenberg's
inequality to the dispersion of
$\widetilde{w}_\alpha(\gamma)$, 
we must consider 
$\widetilde{w}_\alpha(\gamma)$
as a Fourier transform of an
appropriate function. For this
end, we notice that a Mellin
transform $f^\m(s)=\int_0^\infty
f(x)\kern.03cm x^{s-1}\,dx$ is related
to a Fourier transform by
means of the relation
\[
[f(x)]^\m(\sigma+it)=
[f(e^{x})
e^{x\sigma}]^\f
\Big(-\frac{t}{2\pi}\Big).
\]
On the other hand,
we have, with $s=1/2+it$,
\[
\widetilde{w}_\alpha(t)=
\lambda^{\frac{1}{2}-it}
\inte_0^\infty
w_{\alpha,\lambda}(x)\kern.03cm
x^{s-1}\>dx.
\]
Therefore
\[
\widetilde{w}_\alpha(t)=
\lambda^{\frac{1}{2}-it}
[w_{\alpha,\lambda}(
e^x)e^{\frac{1}{2}x}]^\f
\Big(-\frac{t}{2\pi}\Big).
\]
Since the dispersion of
a function does not change
when it is multiplied by a factor
$\lambda^{\frac{1}{2}-it}$,
then we have,
\[
\d[\widetilde{w}_\alpha] =
\d\Big[
[w_{\alpha,\lambda}(
e^x)e^{\frac{1}{2}x}]^\f
\Big(-\frac{t}{2\pi}\Big)\Big].
\]
Now we notice that
because of lemmas~\ref{lema1}
and~\ref{lemaB} we can write,
\[
w_{\alpha,\lambda}^\f=
W^\f\Big\{1+O\Big(
\frac{1}{\alpha^{
\frac{3}{7}}}\Big)\Big\}+
O(e^{-\frac{1}{4}\alpha^{\frac{1}{21}}})=
W^\f\Big\{1+O\Big(
\frac{1}{\alpha^{
\frac{3}{7}}}\Big)\Big\},
\]
as $\alpha\to\infty$. Therefore,
for the dispersion of $\widetilde{w}_\alpha$
we have
\[
\d[\widetilde{w}_\alpha]=
\d[W^\f]\Big\{1+O\Big(
\frac{1}{\alpha^{
\frac{3}{7}}}\Big)\Big\}.
\]
Thus, by
Heisenberg's inequality~(\ref{heisenberg}),
\[
\d[\widetilde{w}_\alpha]
\Big\{1+O\Big(
\frac{1}{\alpha^{
\frac{3}{7}}}\Big)\Big\}
=\d\Big[W^\f
\big(-\frac{t}{2\pi}\big)\Big]
=4\pi^2\kern.05cm\d[W^\f
(t)]\geq
\frac{1}{4\kern.05cm\d[W]}
=\frac{\alpha}{2}-\frac{1}{4}.
\]
Hence,
\[
\d[\widetilde{w}_\alpha]\geq
\Big(
\frac{\alpha}{2}-\frac{1}{4}
\Big)
\Big\{1+O\Big(
\frac{1}{\alpha^{
\frac{3}{7}}}\Big)\Big\}.
\]
This finishes the proof of
theorem~\ref{tres}.


\begin{thebibliography}{99}
\bibitem{Andrews}
Andrews, G.E.; Askey, R.; Roy, R.
Special functions. 
Cambridge University Press, 1999.
\bibitem{Aryan}
Aryan, F.
\emph{On an extension of the Landau-Gonek formula}. 
J. Number Theory 233 (2022), 389-404.
\bibitem{Chandrasekharan}
Chandrasekharan, K. 
\emph{Arithmetical functions}. 
Springer-Verlag, 1970.
\bibitem{Davenport}
Davenport, H. 
\emph{Multiplicative number theory}. 
Third edition. 
Springer-Verlag, 2000.
\bibitem{Fujii}
Fujii, A.
\emph{On a theorem of Landau}.
Proc. Japan Acad. Ser. A Math. Sci. 65 
(1989), no. 2, 51-54.
\bibitem{Gonek2}
Gonek, S.M. 
\emph{An explicit formula of Landau and 
its applications to the theory of the zeta function}.
Contemp. Math. 143 (1993) 395-413.
\bibitem{Igari}
Igari, S.
\emph{Real analysis with an 
introduction to wavelet theory}.
American Mathematical Society, 1998. 
\bibitem{Ingham}
Ingham, A.E.
\emph{The distribution of prime numbers}.
Cambridge University Press, 1990.
\bibitem{Kaczorowski}
Kaczorowski, J.; Languasco, A.; Perelli, A.
\emph{A note on Landau's formula}.
Funct. Approx. Comment. Math. 28 (2000), 173-186.
\bibitem{Landau}
Landau, E. 
\emph{Über die Nullstellen der Zetafunktion}. 
Math. Ann. 71 (1912) 548-564.
\bibitem{Lehman}
Lehman, R.S. 
\emph{On the difference} $\pi(x)-li(x)$. 
Acta Arith. 11 (1966), 397-410.
\bibitem{Montgomery}
Montgomery, H.L.; Vaughan, R.C.
\emph{Multiplicative number theory. 
I. Classical theory}.
Cambridge University Press, Cambridge, 2007.
\bibitem{Oberhettinger}
Oberhettinger, F. 
\emph{Tables of Fourier transforms and 
Fourier transforms of distributions}. 
Springer-Verlag, Berlin, 1990.
\bibitem{Odlyzko}
Odlyzko, A.
\emph{Tables of zeros of the Riemann zeta function}. \\
www.dtc.umn.edu/$\sim$odlyzko/zeta\_tables/index.html.
\bibitem{Mangoldt}
von Mangoldt, H.
\emph{Zu Riemanns Abhandlung 
``Ueber die Anzahl der Primzahlen unter 
einer gegebenen Grösse"}. 
J. Reine Angew. Math. 114 (1895), 255-305.
\end{thebibliography}
\end{document}